\documentclass{article}
\usepackage{amsmath,amsfonts,amssymb,amsthm,mathrsfs,bbm,stmaryrd,array}
\newtheorem{lemma}{Lemma}

\newtheorem{theorem}{Theorem}

\def\mydash{\CJKglue\raise0.2ex\hbox{---\kern-0.01em---}\CJKglue}

\begin{document}
\title{Super congruences involving alternating harmonic sums modulo prime powers\footnotetext{\noindent This work is  supported by the National Natural Science Foundation of China, Project (No.10871169) and the Natural Science Foundation of Zhejiang Province, Project (No. LQ13A010012).}}
\author{\scshape Zhongyan Shen$^{1*}$  Tianxin Cai$^{2\dag}$\\
1 Department of Mathematics, Zhejiang International Study University\\Hangzhou 310012, P.R. China\\
2 Department of Mathematics, Zhejiang University\\ Hangzhou 310027, P.R. China\\
$*$huanchenszyan@163.com\\
 $\dag$txcai@zju.edu.cn}
\date{}
\maketitle
\textbf{Abstract} In 2014, Wang and Cai established the following
harmonic congruence for any odd prime $p$ and positive integer $r$,
\begin{equation*}
 \sum\limits_{i+j+k=p^{r}\atop{i,j,k\in \mathcal{P}_{p}}}\frac{1}{ijk}\equiv-2p^{r-1}B_{p-3} ~(\bmod ~ p^{r}),
\end{equation*}
where $\mathcal{P}_{n}$ denote the set of positive integers which
are prime to $n$.\\
In this note, we  establish a combinational congruence of alternating harmonic
sums  for any odd prime $p$ and positive integers $r$,
\begin{equation*}
 \sum\limits_{i+j+k=p^{r}\atop{i,j,k\in \mathcal{P}_{p}}}\frac{(-1)^{i}}{ijk}
 \equiv \frac{1}{2}p^{r-1}B_{p-3} ~(\bmod ~ p^{r}).
\end{equation*}
For any odd prime $p\geq 5$ and positive integers $r$, we have
\begin{align}
 &4\sum\limits_{i_{1}+i_{2}+i_{3}+i_{4}=2p^{r}\atop{i_{1},~i_{2},~i_{3},~i_{4}\in \mathcal{P}_{p}}}\frac{(-1)^{i_{1}}}{i_{1}i_{2}i_{3}i_{4}}+3\sum\limits_{i_{1}+i_{2}+i_{3}+i_{4}=2p^{r}\atop{i_{1},~i_{2},~i_{3},~i_{4}\in \mathcal{P}_{p}}}\frac{(-1)^{i_{1}+i_{2}}}{i_{1}i_{2}i_{3}i_{4}}
\nonumber\\&\equiv\begin{cases}
 \frac{216}{5}pB_{p-5}\pmod{p^{2}},~if~r=1,\\
 \frac{36}{5}p^{r}B_{p-5}\pmod{p^{r+1}},~if~r>1.\\
  \end{cases}\nonumber
\end{align}
For any odd prime $p> 5$ and positive integers $r$, we have
\begin{align}
 &\sum\limits_{i_{1}+i_{2}+i_{3}+i_{4}+i_{5}=2p^{r}\atop{i_{1},~i_{2},~i_{3},~i_{4},~i_{5}\in \mathcal{P}_{p}}}\frac{(-1)^{i_{1}}}{i_{1}i_{2}i_{3}i_{4}i_{5}}+2\sum\limits_{i_{1}+i_{2}+i_{3}+i_{4}+i_{5}=2p^{r}\atop{i_{1},~i_{2},~i_{3},~i_{4},~i_{5}\in \mathcal{P}_{p}}}\frac{(-1)^{i_{1}+i_{2}}}{i_{1}i_{2}i_{3}i_{4}i_{5}}
\nonumber\\&\equiv\begin{cases}
12B_{p-5}\pmod{p},~if~r=1,\\
6p^{r-1}B_{p-5}\pmod{p^{r}},~if~r>1.\\
  \end{cases}\nonumber
\end{align}\\


\section{Introduction.}
At the beginning of the 21th century, Zhao (Cf.\cite{Zh}) first
announced the following curious congruence involving multiple
harmonic sums for any odd prime $p>3$,
\begin{equation}
 \sum\limits_{i+j+k=p}\frac{1}{ijk}\equiv-2B_{p-3} ~(\bmod ~ p), \label{eq:1}
\end{equation}
which holds when $p=3$ evidently. Here, Bernoulli numbers $B_{k}$
are defined by the recursive relation:
\begin{center}
 $\sum\limits_{i=0}^{n}\binom{n+1}{i}B_{i}=0,n\geq 1 .$
\end{center}
A simple proof of  \eqref{eq:1} was presented in \cite{Ji}.
 This congruence has been generalized
along several directions. First, Zhou and Cai \cite{ZC} established
the following harmonic congruence for prime $p > 3$ and integer
$n\leq p-2$
\begin{equation}
 \sum\limits_{l_{1}+l_{2}+\cdots+l_{n}=p}\frac{1}{l_{1}l_{2}\cdots l_{n}}\equiv\begin{cases}-(n-1)!B_{p-n} ~(\bmod ~ p),~~~ if~ 2\nmid n,\\
 -\frac{n(n!)}{2(n+1)}pB_{p-n-1} ~(\bmod  ~p^{2}), ~~~if~ 2\mid n.\\
 \end{cases}\label{eq:13}
\end{equation}
Later, Xia and Cai \cite{XC} generalized \eqref{eq:1} to
\begin{equation*}
 \sum\limits_{i+j+k=p}\frac{1}{ijk}\equiv-\frac{12B_{p-3}}{p-3}-\frac{3B_{2p-4}}{p-4} ~(\bmod ~ p^{2}),
\end{equation*}
where $p>5$ is a prime. \\
Recently, Wang and Cai \cite{WC} proved for every prime $p\geq 3$
and positive integer $r$,
\begin{equation}
 \sum\limits_{i+j+k=p^{r}\atop{i,j,k\in \mathcal{P}_{p}}}\frac{1}{ijk}\equiv-2p^{r-1}B_{p-3} ~(\bmod ~ p^{r}),\label{eq:2}
\end{equation}
where $\mathcal{P}_{n}$ denote the set of positive integers which
are prime to $n$.\\
Let $n=2$ or 4, for every positive integer $r\geq \frac{n}{2}$ and
prime $p> n$, Zhao \cite{Zh1} generalized \eqref{eq:2} to
\begin{equation}
 \sum\limits_{i_{1}+i_{2}+\cdots+i_{n}=p^{r}\atop{i_{1},i_{2},\cdots,i_{n}\in \mathcal{P}_{p}}}\frac{1}{i_{1}i_{2}\cdots i_{n}}
 \equiv-\frac{n!}{n+1}p^{r}B_{p-n-1} ~(\bmod  ~p^{r+1}).\label{eq:14}
\end{equation}
For any prime $p>5$ and integer $r>1$, Wang \cite{W} proved that
\begin{equation*}
 \sum\limits_{i_{1}+i_{2}+\cdots+i_{5}=p^{r}\atop{i_{1},i_{2},\cdots,i_{5}\in \mathcal{P}_{p}}}\frac{1}{i_{1}i_{2}\cdots i_{5}}
 \equiv-\frac{5!}{6}p^{r-1}B_{p-n-1} ~(\bmod  ~p^{r}).
\end{equation*}

We consider the following  alternating harmonic sums
\begin{equation*}
 \sum\limits_{i_{1}+i_{2}+\cdots+i_{n}=p^{r}\atop{i_{1},i_{2},\cdots,i_{n}\in \mathcal{P}_{p}}}\frac{(\sigma_{1})^{i_{1}}(\sigma_{2})^{i_{2}}
 \cdots(\sigma_{n})^{i_{n}}}{i_{1}i_{2}\cdots i_{n}},
\end{equation*}
where $\sigma_{i}\in \{1,-1\},~i=1,~2,~\cdots,~n$. Given $n$, we only need to consider the following alternating harmonic sums,
\begin{equation*}
 \sum\limits_{i_{1}+i_{2}+\cdots+i_{n}=p^{r}\atop{i_{1},i_{2},\cdots,i_{n}\in \mathcal{P}_{p}}}\frac{(-1)^{i_{1}}}{i_{1}i_{2}\cdots i_{n}},
  \sum\limits_{i_{1}+i_{2}+\cdots+i_{n}=p^{r}\atop{i_{1},i_{2},\cdots,i_{n}\in \mathcal{P}_{p}}}\frac{(-1)^{i_{1}+i_{2}}}{i_{1}i_{2}\cdots i_{n}},\cdots,
  \sum\limits_{i_{1}+i_{2}+\cdots+i_{n}=p^{r}\atop{i_{1},i_{2},\cdots,i_{n}\in \mathcal{P}_{p}}}\frac{(-1)^{i_{1}+i_{2}+\cdots+i_{[\frac{n}{2}]}}}{i_{1}i_{2}\cdots i_{n}}
\end{equation*}
where $[x]$ denote the largest integer less than or equal to $x$.\\

In this paper, we consider the congruences involving the combination of alternating harmonic sums,
\begin{equation*}
 \sum\limits_{i_{1}+i_{2}+\cdots+i_{n}=p^{r}\atop{i_{1},i_{2},\cdots,i_{n}\in \mathcal{P}_{p}}}\frac{(-1)^{i_{1}}}{i_{1}i_{2}\cdots i_{n}},
  \sum\limits_{i_{1}+i_{2}+\cdots+i_{n}=p^{r}\atop{i_{1},i_{2},\cdots,i_{n}\in \mathcal{P}_{p}}}\frac{(-1)^{i_{1}+i_{2}}}{i_{1}i_{2}\cdots i_{n}},\cdots,
  \sum\limits_{i_{1}+i_{2}+\cdots+i_{n}=p^{r}\atop{i_{1},i_{2},\cdots,i_{n}\in \mathcal{P}_{p}}}\frac{(-1)^{i_{1}+i_{2}+\cdots+i_{[\frac{n}{2}]}}}{i_{1}i_{2}\cdots i_{n}}.
\end{equation*}
we obtain the following theorems. Among them, Theorem 1 and
Theorem 2 have been proved by Wang\cite{W1} using different
method, but his method doesn't for Theorem 3 and Theorem 4.

\begin{theorem}
Let $p$ be odd prime and $r$ positive integer, then
 \begin{equation*}
 \sum\limits_{i+j+k=2p^{r}\atop{i,j,k\in \mathcal{P}_{p}}}\frac{(-1)^{i}}{ijk}
 \equiv p^{r-1}B_{p-3} ~(\bmod ~ p^{r}).
\end{equation*}
\end{theorem}

\textbf{Remark 1} There is no solution $(i,~j,~k)$ for the equation
$i+j+k=2p^{r}$ with $i,j,k\in \mathcal{P}_{2p}$.

\begin{theorem}
Let $p$ be odd prime and $r$ positive integer, then
 \begin{equation*}
 \sum\limits_{i+j+k=p^{r}\atop{i,j,k\in \mathcal{P}_{p}}}\frac{(-1)^{i}}{ijk}
 \equiv \frac{1}{2}p^{r-1}B_{p-3} ~(\bmod ~ p^{r}).
\end{equation*}
\end{theorem}

\begin{theorem}
Let $p\geq 5$ be a prime and $r$ positive integer, then
\begin{align}
 &4\sum\limits_{i_{1}+i_{2}+i_{3}+i_{4}=2p^{r}\atop{i_{1},~i_{2},~i_{3},~i_{4}\in \mathcal{P}_{p}}}\frac{(-1)^{i_{1}}}{i_{1}i_{2}i_{3}i_{4}}+3\sum\limits_{i_{1}+i_{2}+i_{3}+i_{4}=2p^{r}\atop{i_{1},~i_{2},~i_{3},~i_{4}\in \mathcal{P}_{p}}}\frac{(-1)^{i_{1}+i_{2}}}{i_{1}i_{2}i_{3}i_{4}}
\nonumber\\&\equiv\begin{cases}
 \frac{216}{5}pB_{p-5}\pmod{p^{2}},~if~r=1,\\
 \frac{36}{5}p^{r}B_{p-5}\pmod{p^{r+1}},~if~r>1.\\
  \end{cases}\nonumber
\end{align}
\end{theorem}

\begin{theorem}
Let $p>5$ be a prime and $r$ positive integer, then
\begin{align}
 &\sum\limits_{i_{1}+i_{2}+i_{3}+i_{4}+i_{5}=2p^{r}\atop{i_{1},~i_{2},~i_{3},~i_{4},~i_{5}\in \mathcal{P}_{p}}}\frac{(-1)^{i_{1}}}{i_{1}i_{2}i_{3}i_{4}i_{5}}+2\sum\limits_{i_{1}+i_{2}+i_{3}+i_{4}+i_{5}=2p^{r}\atop{i_{1},~i_{2},~i_{3},~i_{4},~i_{5}\in \mathcal{P}_{p}}}\frac{(-1)^{i_{1}+i_{2}}}{i_{1}i_{2}i_{3}i_{4}i_{5}}
\nonumber\\&\equiv\begin{cases}
12B_{p-5}\pmod{p},~if~r=1,\\
6p^{r-1}B_{p-5}\pmod{p^{r}},~if~r>1.\\
  \end{cases}\nonumber
\end{align}
\end{theorem}
\section{Preliminaries.}
In order to prove the theorems, we need the following lemmas.
\begin{lemma}[\cite{WC}]
 Let $p$ be odd prime and $r,~m$ positive integers, then
 \begin{equation*}
 \sum\limits_{i+j+k=mp^{r}\atop{i,j,k\in \mathcal{P}_{p}}}\frac{1}{ijk}\equiv-2mp^{r-1}B_{p-3} ~(\bmod  ~ p^{r}).
\end{equation*}
\end{lemma}

\begin{lemma}
 Let $p$ be odd prime and $r,~m$ positive integers, then
 \begin{equation*}
 \sum\limits_{i+j+k=mp^{r}\atop{i,j,k\in \mathcal{P}_{p}}}\frac{1}{ijk}=\frac{6}{mp^{r}}\sum\limits_{1\leq j< l\leq
mp^{r}\atop{j,~l,~l-j\in \mathcal{P}_{p}}}\frac{1}{jl}.
\end{equation*}
\end{lemma}

\begin{proof}
It is easy to see that
\begin{align}
 \sum\limits_{i+j+k=mp^{r}\atop{i,j,k\in \mathcal{P}_{p}}}\frac{1}{ijk}
 =\frac{1}{mp^{r}}\sum\limits_{i+j+k=mp^{r}\atop{i,j,k\in \mathcal{P}_{p}}}\frac{i+j+k}{ijk}
=\frac{3}{mp^{r}}\sum\limits_{i+j+k=mp^{r}\atop{i,j,k\in
\mathcal{P}_{p}}}\frac{1}{ij}.\nonumber
\end{align}
 Let $l=j+k$, then $1\leq j< l\leq mp^{r}$ and $j,~l,~l-j\in
\mathcal{P}_{p}$. By symmetry, we have
\begin{align}
 \frac{3}{mp^{r}}\sum\limits_{i+j+k=mp^{r}\atop{i,j,k\in
\mathcal{P}_{p}}}\frac{1}{ij}=\frac{3}{mp^{r}}\sum\limits_{i+j<
mp^{r}\atop{i,~j,~l\in \mathcal{P}_{p}}}\frac{1}{l}\frac{(i+j)}{ij}
=\frac{6}{mp^{r}}\sum\limits_{1\leq j< l\leq
mp^{r}\atop{j,~l,~l-j\in \mathcal{P}_{p}}}\frac{1}{jl}. \nonumber
\end{align}
This completes the proof of Lemma 2.
\end{proof}
\begin{lemma}
 Let $p>4$ be a prime and $r,~m$ positive integers, then
 \begin{equation*}
 \sum\limits_{i_{1}+i_{2}+i_{3}+i_{4}=mp^{r}\atop{i_{1},~i_{2},~i_{3},~i_{4}\in \mathcal{P}_{p}}}\frac{1}{i_{1}i_{2}i_{3}i_{4}}=\frac{24}{mp^{r}}\sum\limits_{1\leq u_{1}< u_{2}< u_{3}\leq
mp^{r}\atop{u_{1},~u_{3},~u_{2}-u_{1},~u_{3}-u_{2}\in \mathcal{P}_{p}}}\frac{1}{u_{1}u_{2}u_{3}}.
\end{equation*}
\end{lemma}
\begin{proof}
The proof of Lemma 3 is similar to the proof of Lemma 2.
\end{proof}
\begin{lemma}[\cite{ZC}]
 Let $r,~\alpha_{1},~\cdots,~\alpha_{n}$ be positive integers, $r=\alpha_{1}+\cdots+\alpha_{n}\leq p-3$, then
 \begin{equation*}
 \sum\limits_{1\leq l_{1},~\cdots,~l_{n}\leq p-1\atop{l_{i}\neq l_{j},~\forall i\neq j}}\frac{1}{l_{1}^{\alpha_{1}}l_{2}^{\alpha_{2}}\cdots l_{n}^{\alpha_{n}}}\equiv\begin{cases}(-)^{n}(n-1)!\frac{r(r+1)}{2(r+2)}B_{p-r-2}p^{2} ~(\bmod ~ p^{3})~~~ if~ 2\nmid r,\\
 (-)^{n-1}(n-1)!\frac{r}{r+1}B_{p-r-1}p ~(\bmod ~ p^{2}) ~~~if~ 2\mid r.\\
 \end{cases}.
\end{equation*}
\end{lemma}

\begin{lemma}[\cite{W}]
 Let $p$ be odd prime, and $\alpha_{1},~\cdots,~\alpha_{n}$ positive integers, where $r=\alpha_{1}+\cdots+\alpha_{n}\leq p-3$, then
 \begin{equation*}
 \sum\limits_{1\leq l_{1},~\cdots,~l_{n}\leq 2p\atop{l_{i}\neq l_{j},~l_{i}\in \mathcal{P}_{p}}}\frac{1}{l_{1}^{\alpha_{1}}l_{2}^{\alpha_{2}}\cdots l_{n}^{\alpha_{n}}}\equiv\begin{cases}(-)^{n}(n-1)!\frac{2r(r+1)}{r+2}B_{p-r-2}p^{2} ~(\bmod ~ p^{3})~~~ if~ 2\nmid r,\\
 (-)^{n-1}(n-1)!\frac{2r}{r+1}B_{p-r-1}p ~(\bmod ~ p^{2}) ~~~if~ 2\mid r.\\
 \end{cases}.
\end{equation*}
\end{lemma}
\begin{lemma}
 Let $p>4$ be a prime, then
 \begin{equation*}
 \sum\limits_{i_{1}+i_{2}+i_{3}+i_{4}=2p\atop{i_{1},~i_{2},~i_{3},~i_{4}\in \mathcal{P}_{p}}}\frac{1}{i_{1}i_{2}i_{3}i_{4}}\equiv-\frac{240}{5}pB_{p-5} ~(\bmod  ~p^{2}).
\end{equation*}
\end{lemma}
\begin{proof}
By Lemma 2, we have
\begin{equation}
 \sum\limits_{i_{1}+i_{2}+i_{3}+i_{4}=2p\atop{i_{1},~i_{2},~i_{3},~i_{4}\in \mathcal{P}_{p}}}\frac{1}{i_{1}i_{2}i_{3}i_{4}}=\frac{24}{2p}\sum\limits_{1\leq u_{1}< u_{2}< u_{3}\leq
2p\atop{u_{1},~u_{3},~u_{2}-u_{1},~u_{3}-u_{2}\in \mathcal{P}_{p}}}\frac{1}{u_{1}u_{2}u_{3}}.\label{eq:15}
\end{equation}
It is easy to see that
\begin{equation*}
 \sum\limits_{1\leq u_{1}< u_{2}< u_{3}\leq
2p\atop{u_{1},~u_{3},~u_{2}-u_{1},~u_{3}-u_{2}\in \mathcal{P}_{p}}}\frac{1}{u_{1}u_{2}u_{3}}
=\sum\limits_{1\leq u_{1}< u_{2}< u_{3}\leq
2p\atop{u_{1},~u_{2},~u_{3},~u_{2}-u_{1},~u_{3}-u_{2}\in \mathcal{P}_{p}}}\frac{1}{u_{1}u_{2}u_{3}}+\sum\limits_{1\leq u_{1}< p< u_{3}\leq
2p\atop{u_{1},~u_{3},\in \mathcal{P}_{p}}}\frac{1}{u_{1}pu_{3}}.
\end{equation*}
By Lemma 4, we have
\begin{equation}
 \sum\limits_{1\leq u_{1}< p< u_{3}\leq
2p\atop{u_{1},~u_{3},\in \mathcal{P}_{p}}}\frac{1}{u_{1}pu_{3}}=\frac{1}{p}\sum\limits_{1\leq u_{1}< p}\frac{1}{u_{1}}\sum\limits_{p< u_{3}<2p}\frac{1}{u_{3}}
\equiv 0\pmod{p^{3}}.
\end{equation}
Hence
\begin{align}
 &\sum\limits_{1\leq u_{1}< u_{2}< u_{3}\leq
2p\atop{u_{1},~u_{3},~u_{2}-u_{1},~u_{3}-u_{2}\in \mathcal{P}_{p}}}\frac{1}{u_{1}u_{2}u_{3}}
\equiv \sum\limits_{1\leq u_{1}< u_{2}< u_{3}\leq
2p\atop{u_{1},~u_{2},~u_{3},~u_{2}-u_{1},~u_{3}-u_{2}\in \mathcal{P}_{p}}}\frac{1}{u_{1}u_{2}u_{3}}
\nonumber\\&\equiv\sum\limits_{1\leq u_{1}< u_{2}< u_{3}\leq
2p\atop{u_{1},~u_{2},~u_{3}\in \mathcal{P}_{p}}}\frac{1}{u_{1}u_{2}u_{3}}
-\sum\limits_{1\leq u_{1}< u_{1}+p< u_{3}\leq
2p\atop{u_{1},~u_{3}\in \mathcal{P}_{p}}}\frac{1}{u_{1}(u_{1}+p)u_{3}}
\nonumber\\&-\sum\limits_{1\leq u_{1}< u_{2}< u_{2}+p\leq
2p\atop{u_{1},~u_{2}\in \mathcal{P}_{p}}}\frac{1}{u_{1}u_{2}(u_{2}+p)}\pmod{p^{3}}.\nonumber
\end{align}
Replace $u_{3}=u_{2}+p$, then
\begin{align}
 \sum\limits_{1\leq u_{1}< u_{1}+p< u_{3}\leq
2p\atop{u_{1},~u_{3}\in \mathcal{P}_{p}}}\frac{1}{u_{1}(u_{1}+p)u_{3}}
&=\sum\limits_{1\leq u_{1}< u_{1}+p< u_{2}+p\leq
2p\atop{u_{1},~u_{2}\in \mathcal{P}_{p}}}\frac{1}{u_{1}(u_{1}+p)(u_{2}+p)}
\nonumber\\&\equiv\sum\limits_{1\leq u_{1}< u_{2}<p}\frac{1}{u_{1}^{2}u_{2}}(1-\frac{p}{u_{1}}+\frac{p^{2}}{u_{1}^{2}})(1-\frac{p}{u_{2}}+\frac{p^{2}}{u_{2}^{2}})\pmod{p^{3}}.
\nonumber
\end{align}
and
\begin{align}
\sum\limits_{1\leq u_{1}< u_{2}< u_{2}+p\leq
2p\atop{u_{1},~u_{2}\in \mathcal{P}_{p}}}\frac{1}{u_{1}u_{2}(u_{2}+p)}
\equiv\sum\limits_{1\leq u_{1}< u_{2}<p}\frac{1}{u_{1}u_{2}^{2}}(1-\frac{p}{u_{2}}+\frac{p^{2}}{u_{2}^{2}})\pmod{p^{3}}.\nonumber
\end{align}
Thus
\begin{align}
 \sum\limits_{1\leq u_{1}< u_{2}< u_{3}\leq
2p\atop{u_{1},~u_{3},~u_{2}-u_{1},~u_{3}-u_{2}\in \mathcal{P}_{p}}}\frac{1}{u_{1}u_{2}u_{3}}
&\equiv\sum\limits_{1\leq u_{1}< u_{2}< u_{3}\leq
2p\atop{u_{1},~u_{2},~u_{3}\in \mathcal{P}_{p}}}\frac{1}{u_{1}u_{2}u_{3}}
-\sum\limits_{1\leq u_{1}< u_{2}<p}(\frac{1}{u_{1}^{2}u_{2}}
\nonumber\\&+\frac{1}{u_{1}u_{2}^{2}}
-p(\frac{1}{u_{1}^{3}u_{2}}+\frac{1}{u_{1}^{2}u_{2}^{2}}
+\frac{1}{u_{1}u_{2}^{3}})
\nonumber\\&+p^{2}(\frac{1}{u_{1}^{4}u_{2}}+\frac{1}{u_{1}^{3}u_{2}}+\frac{1}{u_{1}^{2}u_{2}^{3}}+\frac{1}{u_{1}^{1}u_{2}^{4}}))
\nonumber\\&\equiv\frac{1}{3!}\sum\limits_{1\leq u_{1},~ u_{2},~ u_{3}\leq
2p\atop{u_{i}\neq u_{j},~u_{i}\in \mathcal{P}_{p}}}\frac{1}{u_{1}u_{2}u_{3}}
-\sum\limits_{1\leq u_{1}, u_{2}<p}(\frac{1}{u_{1}^{2}u_{2}}
\nonumber\\&-p(\frac{1}{u_{1}^{3}u_{2}}+\frac{1}{2}\frac{1}{u_{1}^{2}u_{2}^{2}})
+p^{2}(\frac{1}{u_{1}^{4}u_{2}}+\frac{1}{u_{1}^{3}u_{2}}))\pmod{p^{3}}.\label{eq:16}
\end{align}
 Using Lemma 5 in the first sum of the right hand in \eqref{eq:16}  and using Lemma 4 in the second sum, we have
 \begin{align}
 \sum\limits_{1\leq u_{1}< u_{2}< u_{3}\leq
2p\atop{u_{1},~u_{3},~u_{2}-u_{1},~u_{3}-u_{2}\in \mathcal{P}_{p}}}\frac{1}{u_{1}u_{2}u_{3}}
&\equiv \frac{1}{3!}(-)^{3}(3-1)!\frac{24}{5}B_{p-5}p^{2}-(-)^{2}\frac{12}{10}B_{p-5}p^{2}
\nonumber\\&+\frac{3p}{2}(-\frac{4}{5}B_{p-5}p)-p^{2}\frac{30}{14}B_{p-7}p^{2}
\nonumber\\&\equiv -\frac{20}{5}p^{2}B_{p-5}\pmod{p^{3}}.\label{eq:17}
\end{align}
Combining \eqref{eq:15} with \eqref{eq:17}, we complete the proof
of Lemma 6.
\end{proof}

\begin{lemma}
 Let $p>4$ be odd prime and $r>1$ positive integer, then
 \begin{equation*}
 \sum\limits_{i_{1}+i_{2}+i_{3}+i_{4}=2p^{r}\atop{i_{1},~i_{2},~i_{3},~i_{4}\in \mathcal{P}_{p}}}\frac{1}{i_{1}i_{2}i_{3}i_{4}}\equiv-\frac{48}{5}p^{r}B_{p-5} ~(\bmod  ~p^{r+1}).
\end{equation*}
\end{lemma}
\begin{proof}
The proof of Lemma 7 is similar to the proof method of
\eqref{eq:14} in \cite{Zh1}.
\end{proof}
\begin{lemma}[\cite{W}]
 Let $p>5$ be a prime and $r,~m$ positive integers, $(m,~p)=1$, then\\
 \begin{align}
 &\sum\limits_{i_{1}+i_{2}+\cdots+i_{5}=mp\atop{i_{1},~i_{2},\cdots,~i_{5}\in \mathcal{P}_{p}}}\frac{1}{i_{1}i_{2}i_{3}i_{4}i_{5}}
 \equiv\begin{cases}
-4(5m+m^{3})B_{p-5} ~(\bmod  ~p),~if~r=1,\\
-20mp^{r-1}B_{p-5} ~(\bmod  ~p^{r}),~if~r>1.\\
  \end{cases}\nonumber
\end{align}
\end{lemma}
\begin{lemma}
 Let $p>5$ be a prime and $r,~m$ positive integers, then
 \begin{equation*}
 \sum\limits_{i_{1}+i_{2}+\cdots+i_{5}=mp^{r}\atop{i_{1},~i_{2},~\cdots,~i_{5}\in \mathcal{P}_{p}}}\frac{1}{i_{1}i_{2}i_{3}i_{4}i_{5}}=\frac{120}{mp^{r}}\sum\limits_{1\leq u_{1}< u_{2}< u_{3}< u_{4}\leq
mp^{r}\atop{u_{1},~u_{4},~u_{2}-u_{1},~u_{3}-u_{2},~u_{4}-u_{3}\in \mathcal{P}_{p}}}\frac{1}{u_{1}u_{2}u_{3}u_{4}}.
\end{equation*}
\end{lemma}
\begin{proof}
The proof of Lemma 9 is similar to the proof of Lemma 2.
\end{proof}
\section{Proofs of the Theorems.}
\begin{proof}[Proof of Theorem 1 \\]
It is easy to see that
\begin{align}
 \sum\limits_{i+j+k=2p^{r}\atop{i,j,k\in \mathcal{P}_{p}}}\frac{(-1)^{i}}{ijk}
 &=\frac{1}{2p^{r}}\sum\limits_{i+j+k=2p^{r}\atop{i,j,k\in \mathcal{P}_{p}}}\frac{(-1)^{i}(i+j+k)}{ijk}
\nonumber\\&=\frac{1}{2p^{r}}\sum\limits_{i+j+k=2p^{r}\atop{i,j,k\in
\mathcal{P}_{p}}}(\frac{(-1)^{i}}{jk}+\frac{2(-1)^{i}}{ij}).\label{eq:3}
\end{align}
Let $l=j+k$,then $1\leq j< l\leq 2p^{r}$ and $j,~l,~l-j\in
\mathcal{P}_{p}$, hence
\begin{align}
\sum\limits_{i+j+k=2p^{r}\atop{i,j,k\in
\mathcal{P}_{p}}}\frac{(-1)^{i}}{jk}=\sum\limits_{1\leq j< l\leq
2p^{r}\atop{j,~l,~l-j\in
\mathcal{P}_{p}}}\frac{1}{l}\frac{(-1)^{l}(j+k)}{jk}
=\sum\limits_{1\leq j< l\leq 2p^{r}\atop{j,~l,~l-j\in
\mathcal{P}_{p}}}\frac{2(-1)^{l}}{jl}.\label{eq:4}
\end{align}
Let $l'=i+j$,then $1\leq i< l'\leq 2p^{r}$ and $i,~l',~l'-i\in
\mathcal{P}_{p}$, hence
\begin{align}
\sum\limits_{i+j+k=2p^{r}\atop{i,j,k\in
\mathcal{P}_{p}}}\frac{2(-1)^{i}}{ij}&=\sum\limits_{1\leq i< l'\leq
2p^{r}\atop{i,~l',~l'-i\in
\mathcal{P}_{p}}}\frac{1}{l'}\frac{2(-1)^{i}(i+j)}{ij}
\nonumber\\&=\sum\limits_{1\leq i< l'\leq
2p^{r}\atop{i,~l',~l'-i\in
\mathcal{P}_{p}}}(\frac{2(-1)^{i}}{jl'}+\frac{2(-1)^{i}}{il'}).\label{eq:5}
\end{align}
Noting that $i=l'-j$, $(-1)^{l'-j}=(-1)^{l'+j}$ and we rename $l'$
to $l$, then
\begin{align}
\sum\limits_{1\leq i< l'\leq 2p^{r}\atop{i,~l',~l'-i\in
\mathcal{P}_{p}}}\frac{2(-1)^{i}}{jl'}=\sum\limits_{1\leq j< l\leq
2p^{r}\atop{j,~l,~l-j\in
\mathcal{P}_{p}}}\frac{2(-1)^{j+l}}{jl}.\label{eq:6}
\end{align}
Rename $i$ to $j$ and $l'$ to $l$, then
\begin{align}
\sum\limits_{1\leq i< l'\leq 2p^{r}\atop{i,~l',~l'-i\in
\mathcal{P}_{p}}}\frac{2(-1)^{i}}{il'}=\sum\limits_{1\leq j< l\leq
2p^{r}\atop{j,~l,~l-j\in
\mathcal{P}_{p}}}\frac{2(-1)^{j}}{jl}.\label{eq:7}
\end{align}
   Combining \eqref{eq:3}-\eqref{eq:7}, we have
\begin{align}
 \sum\limits_{i+j+k=2p^{r}\atop{i,j,k\in \mathcal{P}_{p}}}\frac{(-1)^{i}}{ijk}
 &=\frac{1}{p^{r}}\sum\limits_{1\leq j< l\leq
2p^{r}\atop{j,~l,~l-j\in
\mathcal{P}_{p}}}(\frac{(-1)^{l}}{jl}+\frac{(-1)^{j+l}}{jl}+\frac{(-1)^{j}}{jl})
\nonumber\\&=\frac{1}{p^{r}}\sum\limits_{1\leq j< l\leq
2p^{r}\atop{j,~l,~l-j\in\mathcal{P}_{p}}}\frac{(1+(-1)^{l})(1+(-1)^{j})}{jl}-\frac{1}{p^{r}}\sum\limits_{1\leq
j< l\leq 2p^{r}\atop{j,~l,~l-j\in\mathcal{P}_{p}}}\frac{1}{jl}
\nonumber\\&=\frac{1}{p^{r}}\sum\limits_{1\leq j< l\leq
2p^{r}\atop{j,~l,~l-j\in\mathcal{P}_{p}},j even,~l
even}\frac{4}{jl}-\frac{1}{p^{r}}\sum\limits_{1\leq j< l\leq
2p^{r}\atop{j,~l,~l-j\in\mathcal{P}_{p}}}\frac{1}{jl}.\label{eq:8}
\end{align}
Let $j=2j',~l=2l'$ in the first sum of \eqref{eq:8} and noting that
\begin{align}
\sum\limits_{1\leq j'< l'\leq
p^{r}\atop{j',~l',~l'-j'\in\mathcal{P}_{p}}}\frac{1}{j'l'}=\sum\limits_{1\leq j< l\leq
p^{r}\atop{j,~l,~l-j\in\mathcal{P}_{p}}}\frac{1}{jl},\nonumber
\end{align}
\eqref{eq:8} is equal to
\begin{align}
 \sum\limits_{i+j+k=2p^{r}\atop{i,j,k\in \mathcal{P}_{p}}}\frac{(-1)^{i}}{ijk}
 =\frac{1}{p^{r}}\sum\limits_{1\leq j< l\leq
p^{r}\atop{j,~l,~l-j\in\mathcal{P}_{p}}}\frac{1}{jl}-\frac{1}{p^{r}}\sum\limits_{1\leq j< l\leq
2p^{r}\atop{j,~l,~l-j\in\mathcal{P}_{p}}}\frac{1}{jl}.\label{eq:9}
\end{align}
By Lemma 1, Lemma 2 and \eqref{eq:9}, we obtain
\begin{align}
 \sum\limits_{i+j+k=2p^{r}\atop{i,j,k\in \mathcal{P}_{p}}}\frac{(-1)^{i}}{ijk}
 &=\frac{1}{p^{r}}\sum\limits_{1\leq j< l\leq
p^{r}\atop{j,~l,~l-j\in\mathcal{P}_{p}}}\frac{1}{jl}-\frac{1}{p^{r}}\sum\limits_{1\leq j< l\leq
2p^{r}\atop{j,~l,~l-j\in\mathcal{P}_{p}}}\frac{1}{jl}
\nonumber\\&\equiv p^{r-1}B_{p-3} ~~(\bmod  ~ p^{r}).\nonumber
\end{align}
This completes the proof of Theorem 1.
\end{proof}
\begin{proof}[Proof of Theorem 2 \\]
For every triple $(i,~j,~k)$ of positive integers which satisfies
$i+j+k=2p^{r},~i,j,k\in \mathcal{P}_{p}$, we take it to 3 cases.\\

Cases 1. If $1\leq i,~j,~k\leq p^{r}-1$ are coprime to $pq$,
$(i,~j,~k)\leftrightarrow (p^{r}-i,~p^{r}-j,~p^{r}-k)$ is a
bijection between the solutions of $i+j+k=2p^{r}$ and
$i+j+k=p^{r},~i,j,k\in \mathcal{P}_{p}$, we have
 \begin{align}
 \sum_{\substack{i+j+k=2p^{r}\\i,j,k\in \mathcal{P}_{p}\\1\leq i,~j,~k\leq p^{r}-1}}\frac{(-1)^{i}}{ijk}
 &\equiv \sum\limits_{i+j+k=p^{r}\atop{i,j,k\in \mathcal{P}_{p}}}\frac{(-1)^{p^{r}-i}}{(p^{r}-i)(p^{r}-j)(p^{r}-k)}
\nonumber\\& \equiv \sum\limits_{i+j+k=p^{r}\atop{i,j,k\in
\mathcal{P}_{p}}}\frac{(-1)^{i}}{ijk}\pmod{p^{r}}.\label{eq:10}
\end{align}
Cases 2. If $p^{r}+1\leq i\leq 2p^{r}-1,~1\leq j,~k\leq p^{r}-1$
are coprime to $p$, $(i,~j,~k)\leftrightarrow (p^{r}+i,~j,~k)$ is
a bijection between the solutions of $i+j+k=2p^{r}$ and
$i+j+k=p^{r},~i,j,k\in \mathcal{P}_{p}$, we have
 \begin{align}
 \sum_{\substack{i+j+k=2p^{r}\\i,j,k\in \mathcal{P}_{p}\\p^{r}+1\leq i\leq 2p^{r}-1,1\leq j,~k\leq p^{r}-1}}\frac{(-1)^{i}}{ijk}
 &\equiv \sum\limits_{i+j+k=p^{r}\atop{i,j,k\in \mathcal{P}_{p}}}\frac{(-1)^{p^{r}+i}}{(p^{r}+i)jk}
\nonumber\\& \equiv -\sum\limits_{i+j+k=p^{r}\atop{i,j,k\in
\mathcal{P}_{p}}}\frac{(-1)^{i}}{ijk}\pmod{p^{r}}.\label{eq:11}
\end{align}
Cases 3. If $p^{r}+1\leq j\leq 2p^{r}-1,~1\leq i,~k\leq p^{r}-1$
or $p^{r}+1\leq k\leq 2p^{r}-1,1\leq i,~j\leq p^{r}-1$are coprime
to $p$, $(i,~j,~k)\leftrightarrow (i,~p^{r}+j,~k)$ in the former
and $(i,~j,~k)\leftrightarrow (i,~j,~p^{r}+k)$ in the later are
the bijections between the solutions of $i+j+k=2p^{r}$ and
$i+j+k=p^{r},~i,j,k\in \mathcal{P}_{p}$, we have
 \begin{align}
 &\sum_{\substack{i+j+k=2p^{r}\\i,j,k\in \mathcal{P}_{p}\\p^{r}+1\leq j\leq 2p^{r}-1,1\leq i,~k\leq p^{r}-1}}\frac{(-1)^{i}}{ijk}
+\sum_{\substack{i+j+k=2p^{r}\\i,j,k\in \mathcal{P}_{p}\\p^{r}+1\leq
k\leq 2p^{r}-1,1\leq i,~j\leq p^{r}-1}}\frac{(-1)^{i}}{ijk}
 \nonumber\\&\equiv \sum\limits_{i+j+k=p^{r}\atop{i,j,k\in
 \mathcal{P}_{p}}}\frac{(-1)^{i}}{i(p^{r}+j)k}+\sum\limits_{i+j+k=p^{r}\atop{i,j,k\in
 \mathcal{P}_{p}}}\frac{(-1)^{i}}{ij(p^{r}+k)}
\nonumber\\& \equiv 2\sum\limits_{i+j+k=p^{r}\atop{i,j,k\in
\mathcal{P}_{p}}}\frac{(-1)^{i}}{ijk}\pmod{p^{r}}.\label{eq:12}
\end{align}
Combining \eqref{eq:10}-\eqref{eq:12}, we have
 \begin{align}
 \sum\limits_{i+j+k=2p^{r}\atop{i,j,k\in \mathcal{P}_{p}}}\frac{(-1)^{i}}{ijk}&=
 \sum_{\substack{i+j+k=2p^{r}\\i,j,k\in \mathcal{P}_{p}\\1\leq i,~j,~k\leq p^{r}-1}}\frac{(-1)^{i}}{ijk}
+ \sum_{\substack{i+j+k=2p^{r}\\i,j,k\in
\mathcal{P}_{p}\\p^{r}+1\leq i\leq 2p^{r}-1,1\leq j,~k\leq
p^{r}-1}}\frac{(-1)^{i}}{ijk}
\nonumber\\&+\sum_{\substack{i+j+k=2p^{r}\\i,j,k\in
\mathcal{P}_{p}\\p^{r}+1\leq j\leq 2p^{r}-1,1\leq i,~k\leq
p^{r}-1}}\frac{(-1)^{i}}{ijk}
+\sum_{\substack{i+j+k=2p^{r}\\i,j,k\in
\mathcal{P}_{p}\\p^{r}+1\leq k\leq 2p^{r}-1,1\leq i,~j\leq
p^{r}-1}}\frac{(-1)^{i}}{ijk}
 \nonumber\\&\equiv 2\sum\limits_{i+j+k=p^{r}\atop{i,j,k\in
\mathcal{P}_{p}}}\frac{(-1)^{i}}{ijk}\pmod{p^{r}}.\nonumber
\end{align}
By Theorem 1, we complete the proof of Theorem 2.
\end{proof}

\begin{proof}[Proof of Theorem 3 \\]
By symmetry, it is easy to see that
 \begin{align}
 &\sum\limits_{i_{1}+i_{2}+i_{3}+i_{4}=2p^{r}\atop{i_{1},~i_{2},~i_{3},~i_{4}\in \mathcal{P}_{p}}}\frac{(-1)^{i_{1}}}{i_{1}i_{2}i_{3}i_{4}}
 =\frac{1}{2p^{r}}\sum\limits_{i_{1}+i_{2}+i_{3}+i_{4}=2p^{r}\atop{i_{1},~i_{2},~i_{3},~i_{4}\in \mathcal{P}_{p}}}\frac{(-1)^{i_{1}}(i_{1}+i_{2}+i_{3}+i_{4})}{i_{1}i_{2}i_{3}i_{4}}
\nonumber\\&=\frac{1}{2p^{r}}\sum\limits_{i_{1}+i_{2}+i_{3}+i_{4}=2p^{r}\atop{i_{1},~i_{2},~i_{3},~i_{4}\in
\mathcal{P}_{p}}}(\frac{(-1)^{i_{1}}}{i_{2}i_{3}i_{4}}+3\frac{(-1)^{i_{1}}}{i_{1}i_{3}i_{4}}).\label{eq:19}
\end{align}
Let $u_{3}=i_{2}+i_{3}+i_{4}$ in the first sum of the last
equation in \eqref{eq:19}, then $i_{1}=2p^{r}-u_{3}$,
\eqref{eq:19} equals to
\begin{align}
\nonumber\\&=\frac{1}{2p^{r}}[\sum\limits_{u_{3}=i_{2}+i_{3}+i_{4}<2p^{r}\atop{u_{3},~i_{2},~i_{3},~i_{4}\in
\mathcal{P}_{p}}}\frac{(-1)^{2p^{r}-u_{3}}(i_{2}+i_{3}+i_{4})}{i_{2}i_{3}i_{4}u_{3}}
\nonumber\\&~~~~~~+3\sum\limits_{u_{3}=i_{1}+i_{3}+i_{4}<2p^{r}\atop{u_{3},~i_{1},~i_{3},~i_{4}\in
\mathcal{P}_{p}}}\frac{(-1)^{i_{1}}(i_{1}+i_{3}+i_{4})}{i_{1}i_{3}i_{4}u_{3}}]
\nonumber\\&=\frac{1}{2p^{r}}[3\sum\limits_{u_{3}=i_{2}+i_{3}+i_{4}<2p^{r}\atop{u_{3},~i_{2},~i_{3},~i_{4}\in
\mathcal{P}_{p}}}\frac{(-1)^{u_{3}}}{i_{3}i_{4}u_{3}}
+3\sum\limits_{u_{3}=i_{1}+i_{3}+i_{4}<2p^{r}\atop{u_{3},~i_{1},~i_{3},~i_{4}\in
\mathcal{P}_{p}}}\frac{(-1)^{i_{1}}}{i_{3}i_{4}u_{3}}
\nonumber\\&~~~~~~+6\sum\limits_{u_{3}=i_{1}+i_{3}+i_{4}<2p^{r}\atop{u_{3},~i_{1},~i_{3},~i_{4}\in
\mathcal{P}_{p}}}\frac{(-1)^{i_{1}}}{i_{1}i_{3}u_{3}}].\label{eq:20}
\end{align}
Let $u_{2}=i_{3}+i_{4}$ in the second  sum of the last equation in
\eqref{eq:20}, since $u_{3}=i_{1}+i_{3}+i_{4}$ , then
$i_{1}=u_{3}-u_{2}$, \eqref{eq:20} equals to

\begin{align}
\nonumber\\&=\frac{1}{2p^{r}}[3\sum\limits_{u_{2}=i_{3}+i_{4}<u_{3}=<2p^{r}\atop{u_{3},~u_{3}-u_{2},~i_{3},~i_{4}\in
\mathcal{P}_{p}}}\frac{(-1)^{u_{3}}(i_{3}+i_{4})}{i_{3}i_{4}u_{2}u_{3}}
+3\sum\limits_{u_{2}=i_{3}+i_{4}<u_{3}=<2p^{r}\atop{u_{3},~u_{3}-u_{2},~i_{3},~i_{4}\in
\mathcal{P}_{p}}}\frac{(-1)^{u_{3}-u_{2}}(i_{3}+i_{4})}{i_{3}i_{4}u_{3}}
\nonumber\\&~~~~~~+6\sum\limits_{u_{2}=i_{1}+i_{3}<u_{3}=<2p^{r}\atop{u_{3},~u_{3}-u_{2},~i_{1},~i_{3}\in
\mathcal{P}_{p}}}\frac{(-1)^{i_{1}}(i_{1}+i_{3})}{i_{1}i_{3}u_{2}u_{3}}
\nonumber\\&=\frac{1}{2p^{r}}[6\sum\limits_{0<u_{1}<u_{2}<u_{3}=<2p^{r}\atop{u_{1},~u_{3},~u_{3}-u_{2},~u_{2}-u_{1}\in
\mathcal{P}_{p}}}\frac{(-1)^{u_{3}}}{u_{1}u_{2}u_{3}}
+6\sum\limits_{0<u_{1}<u_{2}<u_{3}=<2p^{r}\atop{u_{1},~u_{3},~u_{3}-u_{2},~u_{2}-u_{1}\in
\mathcal{P}_{p}}}\frac{(-1)^{u_{3}-u_{2}}}{u_{1}u_{2}u_{3}}
\nonumber\\&~~~~~~+6\sum\limits_{0<u_{1}<u_{2}<u_{3}=<2p^{r}\atop{u_{1},~u_{3},~u_{3}-u_{2},~u_{2}-u_{1}\in
\mathcal{P}_{p}}}\frac{(-1)^{u_{2}-u_{1}}}{u_{1}u_{2}u_{3}}+6\sum\limits_{0<u_{1}<u_{2}<u_{3}=<2p^{r}\atop{u_{1},~u_{3},~u_{3}-u_{2},~u_{2}-u_{1}\in
\mathcal{P}_{p}}}\frac{(-1)^{u_{1}}}{u_{1}u_{2}u_{3}}
\nonumber\\&=\frac{3}{p^{r}}\sum\limits_{0<u_{1}<u_{2}<u_{3}=<2p^{r}\atop{u_{1},~u_{3},~u_{3}-u_{2},~u_{2}-u_{1}\in
\mathcal{P}_{p}}}\frac{(-1)^{u_{3}}+(-1)^{u_{2}+u_{3}}+(-1)^{u_{1}+u_{2}}+(-1)^{u_{1}}}{u_{1}u_{2}u_{3}}.\nonumber
\end{align}
Similarly, we have
 \begin{align}
 &\sum\limits_{i_{1}+i_{2}+i_{3}+i_{4}=2p^{r}\atop{i_{1},~i_{2},~i_{3},~i_{4}\in \mathcal{P}_{p}}}\frac{(-1)^{i_{1}+i_{2}}}{i_{1}i_{2}i_{3}i_{4}}
\nonumber\\& =\frac{4}{p^{r}}\sum\limits_{0<u_{1}<u_{2}<u_{3}=<2p^{r}\atop{u_{1},~u_{3},~u_{3}-u_{2},~u_{2}-u_{1}\in \mathcal{P}_{p}}}\frac{(-1)^{u_{2}}+(-1)^{u_{1}+u_{2}+u_{3}}+(-1)^{u_{1}+u_{3}}}{u_{1}u_{2}u_{3}}.\nonumber
\end{align}
Hence
\begin{align}
 &4\sum\limits_{i_{1}+i_{2}+i_{3}+i_{4}=2p^{r}\atop{i_{1},~i_{2},~i_{3},~i_{4}\in \mathcal{P}_{p}}}\frac{(-1)^{i_{1}}}{i_{1}i_{2}i_{3}i_{4}}+3\sum\limits_{i_{1}+i_{2}+i_{3}+i_{4}=2p^{r}\atop{i_{1},~i_{2},~i_{3},~i_{4}\in \mathcal{P}_{p}}}\frac{(-1)^{i_{1}+i_{2}}}{i_{1}i_{2}i_{3}i_{4}}
\nonumber\\& =\frac{12}{p^{r}}\sum\limits_{0<u_{1}<u_{2}<u_{3}=<2p^{r}\atop{u_{1},~u_{3},~u_{3}-u_{2},~u_{2}-u_{1}\in \mathcal{P}_{p}}}\frac{[1+(-1)^{u_{1}}][1+(-1)^{u_{2}}][1+(-1)^{u_{3}}]-1}{u_{1}u_{2}u_{3}}
\nonumber\\& =\frac{12}{p^{r}}\sum\limits_{0<u_{1}<u_{2}<u_{3}=<p^{r}\atop{u_{1},~u_{3},~u_{3}-u_{2},~u_{2}-u_{1}\in \mathcal{P}_{p}}}\frac{1}{u_{1}u_{2}u_{3}}-\frac{12}{p^{r}}\sum\limits_{0<u_{1}<u_{2}<u_{3}=<2p^{r}\atop{u_{1},~u_{3},~u_{3}-u_{2},~u_{2}-u_{1}\in \mathcal{P}_{p}}}\frac{1}{u_{1}u_{2}u_{3}}.\nonumber
\end{align}
By Lemma 3, we have
\begin{align}
 &4\sum\limits_{i_{1}+i_{2}+i_{3}+i_{4}=2p^{r}\atop{i_{1},~i_{2},~i_{3},~i_{4}\in \mathcal{P}_{p}}}\frac{(-1)^{i_{1}}}{i_{1}i_{2}i_{3}i_{4}}+3\sum\limits_{i_{1}+i_{2}+i_{3}+i_{4}=2p^{r}\atop{i_{1},~i_{2},~i_{3},~i_{4}\in \mathcal{P}_{p}}}\frac{(-1)^{i_{1}+i_{2}}}{i_{1}i_{2}i_{3}i_{4}}
\nonumber\\& =\frac{1}{2}\sum\limits_{i_{1}+i_{2}+i_{3}+i_{4}=p^{r}\atop{i_{1},~i_{2},~i_{3},~i_{4}\in \mathcal{P}_{p}}}\frac{1}{i_{1}i_{2}i_{3}i_{4}}
-\sum\limits_{i_{1}+i_{2}+i_{3}+i_{4}=2p^{r}\atop{i_{1},~i_{2},~i_{3},~i_{4}\in \mathcal{P}_{p}}}\frac{1}{i_{1}i_{2}i_{3}i_{4}}.\nonumber
\end{align}
By \eqref{eq:13} and Lemma 6, we have
\begin{align}
 &4\sum\limits_{i_{1}+i_{2}+i_{3}+i_{4}=2p\atop{i_{1},~i_{2},~i_{3},~i_{4}\in \mathcal{P}_{p}}}\frac{(-1)^{i_{1}}}{i_{1}i_{2}i_{3}i_{4}}+3\sum\limits_{i_{1}+i_{2}+i_{3}+i_{4}=2p\atop{i_{1},~i_{2},~i_{3},~i_{4}\in \mathcal{P}_{p}}}\frac{(-1)^{i_{1}+i_{2}}}{i_{1}i_{2}i_{3}i_{4}}
\nonumber\\& \equiv-\frac{24}{5}pB_{p-5}
+\frac{240}{5}pB_{p-5}\equiv\frac{216}{5}pB_{p-5}\pmod{p^{2}}.\nonumber
\end{align}
By \eqref{eq:14} and Lemma 7, if $r\geq 2$, then
\begin{align}
 &4\sum\limits_{i_{1}+i_{2}+i_{3}+i_{4}=2p^{r}\atop{i_{1},~i_{2},~i_{3},~i_{4}\in \mathcal{P}_{p}}}\frac{(-1)^{i_{1}}}{i_{1}i_{2}i_{3}i_{4}}+3\sum\limits_{i_{1}+i_{2}+i_{3}+i_{4}=2p^{r}\atop{i_{1},~i_{2},~i_{3},~i_{4}\in \mathcal{P}_{p}}}\frac{(-1)^{i_{1}+i_{2}}}{i_{1}i_{2}i_{3}i_{4}}
\nonumber\\& \equiv-\frac{12}{5}p^{r}B_{p-5}
+\frac{48}{5}p^{r}B_{p-5}\equiv\frac{36}{5}p^{r}B_{p-5}\pmod{p^{r+1}}.\nonumber
\end{align}
This completes the proof of Theorem 3.
\end{proof}

\begin{proof}[Proof of Theorem 4 \\]
Similar to the proofs of Theorem 1 and Theorem 3, we have
 \begin{align}
 \sum\limits_{i_{1}+i_{2}+i_{3}+i_{4}+i_{5}=2p^{r}\atop{i_{1},~i_{2},~i_{3},~i_{4},~i_{5}\in \mathcal{P}_{p}}}\frac{(-1)^{i_{1}}}{i_{1}i_{2}i_{3}i_{4}i_{5}}
&=\frac{12}{p^{r}}\sum\limits_{0<u_{1}<u_{2}<u_{3}<u_{4}=<2p^{r}\atop{u_{1},~u_{4},~u_{2}-u_{1},~u_{3}-u_{2},~u_{4}-u_{3}\in
\mathcal{P}_{p}}}
 \frac{1}{u_{1}u_{2}u_{3}u_{4}}((-1)^{u_{1}}
  \nonumber\\&+(-1)^{u_{4}}+(-1)^{u_{1}+u_{2}}+(-1)^{u_{2}+u_{3}}+(-1)^{u_{3}+u_{4}})\nonumber
\end{align}
and
\begin{align}
 &\sum\limits_{i_{1}+i_{2}+i_{3}+i_{4}+i_{5}=2p^{r}\atop{i_{1},~i_{2},~i_{3},~i_{4},~i_{5}\in \mathcal{P}_{p}}}\frac{(-1)^{i_{1}+i_{2}}}{i_{1}i_{2}i_{3}i_{4}i_{5}}
 =\frac{6}{p^{r}}\sum\limits_{0<u_{1}<u_{2}<u_{3}<u_{4}=<2p^{r}\atop{u_{1},~u_{4},~u_{2}-u_{1},~u_{3}-u_{2},~u_{4}-u_{3}\in \mathcal{P}_{p}}}
 \frac{1}{u_{1}u_{2}u_{3}u_{4}}((-1)^{u_{2}}
 \nonumber\\&+(-1)^{u_{3}}+(-1)^{u_{1}+u_{3}}+(-1)^{u_{1}+u_{4}}+(-1)^{u_{2}+u_{4}}+(-1)^{u_{1}+u_{2}+u_{3}}
 \nonumber\\&+(-1)^{u_{1}+u_{2}+u_{4}}+(-1)^{u_{1}+u_{3}+u_{4}}+(-1)^{u_{2}+u_{3}+u_{4}}+(-1)^{u_{1}+u_{2}+u_{3}+u_{4}})
\nonumber
\end{align}
Hence
\begin{align}
 &\sum\limits_{i_{1}+i_{2}+i_{3}+i_{4}+i_{5}=2p^{r}\atop{i_{1},~i_{2},~i_{3},~i_{4},~i_{5}\in \mathcal{P}_{p}}}\frac{(-1)^{i_{1}}}{i_{1}i_{2}i_{3}i_{4}i_{5}}+2\sum\limits_{i_{1}+i_{2}+i_{3}+i_{4}+i_{5}=2p^{r}\atop{i_{1},~i_{2},~i_{3},~i_{4},~i_{5}\in \mathcal{P}_{p}}}\frac{(-1)^{i_{1}+i_{2}}}{i_{1}i_{2}i_{3}i_{4}i_{5}}
\nonumber\\& =\frac{12}{p^{r}}\sum\limits_{0<u_{1}<u_{2}<u_{3}<u_{4}=<2p^{r}\atop{u_{1},~u_{4},~u_{2}-u_{1},~u_{3}-u_{2},~u_{4}-u_{3}\in \mathcal{P}_{p}}}\frac{[1+(-1)^{u_{1}}][1+(-1)^{u_{2}}][1+(-1)^{u_{3}}][1+(-1)^{u_{4}}]-1}{u_{1}u_{2}u_{3}u_{4}}
\nonumber\\& =\frac{12}{p^{r}}\sum\limits_{0<u_{1}<u_{2}<u_{3}<u_{4}=<p^{r}\atop{u_{1},~u_{4},~u_{2}-u_{1},~u_{3}-u_{2},~u_{4}-u_{3}\in \mathcal{P}_{p}}}\frac{1}{u_{1}u_{2}u_{3}u_{4}}-\frac{12}{p^{r}}\sum\limits_{0<u_{1}<u_{2}<u_{3}<u_{4}=<2p^{r}\atop{u_{1},~u_{4},~u_{2}-u_{1},~u_{3}-u_{2},~u_{4}-u_{3}\in \mathcal{P}_{p}}}\frac{1}{u_{1}u_{2}u_{3}u_{4}}.\nonumber
\end{align}
By Lemma 9, we have
 \begin{align}
 &\sum\limits_{i_{1}+i_{2}+i_{3}+i_{4}+i_{5}=2p^{r}\atop{i_{1},~i_{2},~i_{3},~i_{4},~i_{5}\in \mathcal{P}_{p}}}\frac{(-1)^{i_{1}}}{i_{1}i_{2}i_{3}i_{4}i_{5}}+2\sum\limits_{i_{1}+i_{2}+i_{3}+i_{4}+i_{5}=2p^{r}\atop{i_{1},~i_{2},~i_{3},~i_{4},~i_{5}\in \mathcal{P}_{p}}}\frac{(-1)^{i_{1}+i_{2}}}{i_{1}i_{2}i_{3}i_{4}i_{5}}
\nonumber\\& =\frac{1}{10}\sum\limits_{i_{1}+i_{2}+i_{3}+i_{4}+i_{5}=p^{r}\atop{i_{1},~i_{2},~i_{3},~i_{4},~i_{5}\in \mathcal{P}_{p}}}\frac{1}{i_{1}i_{2}i_{3}i_{4}i_{5}}
-\frac{2}{10}\sum\limits_{i_{1}+i_{2}+i_{3}+i_{4}+i_{5}=2p^{r}\atop{i_{1},~i_{2},~i_{3},~i_{4},~i_{5}\in \mathcal{P}_{p}}}\frac{1}{i_{1}i_{2}i_{3}i_{4}i_{5}}.\nonumber
\end{align}
By \eqref{eq:13} and Lemma 8(1), we have
 \begin{align}
 &\sum\limits_{i_{1}+i_{2}+i_{3}+i_{4}+i_{5}=2p\atop{i_{1},~i_{2},~i_{3},~i_{4},~i_{5}\in \mathcal{P}_{p}}}\frac{(-1)^{i_{1}}}{i_{1}i_{2}i_{3}i_{4}i_{5}}+2\sum\limits_{i_{1}+i_{2}+i_{3}+i_{4}+i_{5}=2p\atop{i_{1},~i_{2},~i_{3},~i_{4},~i_{5}\in \mathcal{P}_{p}}}\frac{(-1)^{i_{1}+i_{2}}}{i_{1}i_{2}i_{3}i_{4}i_{5}}
\nonumber\\& \equiv-\frac{24}{10}B_{p-5}
+\frac{144}{10}B_{p-5}\equiv12B_{p-5}\pmod{p}.\nonumber
\end{align}
By  Lemma 8(2), if $r\geq 2$, then
\begin{align}
 &\sum\limits_{i_{1}+i_{2}+i_{3}+i_{4}+i_{5}=2p^{r}\atop{i_{1},~i_{2},~i_{3},~i_{4},~i_{5}\in \mathcal{P}_{p}}}\frac{(-1)^{i_{1}}}{i_{1}i_{2}i_{3}i_{4}i_{5}}+2\sum\limits_{i_{1}+i_{2}+i_{3}+i_{4}+i_{5}=2p^{r}\atop{i_{1},~i_{2},~i_{3},~i_{4},~i_{5}\in \mathcal{P}_{p}}}\frac{(-1)^{i_{1}+i_{2}}}{i_{1}i_{2}i_{3}i_{4}i_{5}}
\nonumber\\& \equiv-2p^{r-1}B_{p-5}
+8p^{r-1}B_{p-5}\equiv6p^{r-1}B_{p-5}\pmod{p^{r}}.\nonumber
\end{align}
This completes the proof of Theorem 4.
\end{proof}
\textbf{Remark 2}~ Let $p$ be odd prime and $r,~m$ positive
integers, $(m,p)=1$, using Lemma 1 and Lemma 2, similar to the
proof of Theorem 1, we can prove that
 \begin{equation*}
  \sum\limits_{i+j+k=2mp^{r}\atop{i,j,k\in \mathcal{P}_{p}}}\frac{(-1)^{i}}{ijk}\equiv mp^{r-1}B_{p-3} ~(\bmod ~ p^{r}).
\end{equation*}

In particular, if $m=1$, it becomes Theorem 1.\\\

 Let $p$ be odd prime and $r,~m$ positive
integers, $(m,p)=1$, similar to the proof of Theorem 2, we can
prove that
 \begin{equation*}
  \sum\limits_{i+j+k=mp^{r}\atop{i,j,k\in \mathcal{P}_{p}}}\frac{(-1)^{i}}{ijk}
  \equiv \frac{1}{2}\sum\limits_{i+j+k=2mp^{r}\atop{i,j,k\in \mathcal{P}_{p}}}\frac{(-1)^{i}}{ijk}\equiv \frac{m}{2}p^{r-1}B_{p-3} ~(\bmod ~ p^{r}).
\end{equation*}

In particular, if $m=1$, it becomes Theorem 2.\\

Let $p>4$ be a prime and $r,~m$ positive integers, $(m,p)=1$, we
can deduce the congruence for
\begin{align}
 4\sum\limits_{i_{1}+i_{2}+i_{3}+i_{4}=2mp^{r}\atop{i_{1},~i_{2},~i_{3},~i_{4}\in \mathcal{P}_{p}}}\frac{(-1)^{i_{1}}}{i_{1}i_{2}i_{3}i_{4}}+3\sum\limits_{i_{1}+i_{2}+i_{3}+i_{4}=2mp^{r}\atop{i_{1},~i_{2},~i_{3},~i_{4}\in
  \mathcal{P}_{p}}}\frac{(-1)^{i_{1}+i_{2}}}{i_{1}i_{2}i_{3}i_{4}} \pmod{p^{r+1}}.\nonumber
\end{align}
Let $p>5$ be a prime and $r,~m$ positive integers, $(m,p)=1$, we
can deduce the congruence for
\begin{align}
 \sum\limits_{i_{1}+i_{2}+i_{3}+i_{4}+i_{5}=2mp^{r}\atop{i_{1},~i_{2},~i_{3},~i_{4},~i_{5}\in \mathcal{P}_{p}}}\frac{(-1)^{i_{1}}}{i_{1}i_{2}i_{3}i_{4}i_{5}}+2\sum\limits_{i_{1}+i_{2}+i_{3}+i_{4}+i_{5}=2mp^{r}\atop{i_{1},~i_{2},~i_{3},~i_{4},~i_{5}\in \mathcal{P}_{p}}}\frac{(-1)^{i_{1}+i_{2}}}{i_{1}i_{2}i_{3}i_{4}i_{5}}\pmod{p^{r}}.\nonumber
\end{align}

Similarly, we can consider the congruence
\begin{equation*}
 \sum\limits_{i_{1}+i_{2}+\cdots+i_{n}=p^{r}\atop{i_{1},i_{2},\cdots,i_{n}\in \mathcal{P}_{p}}}\frac{(\sigma_{1})^{i_{1}}(\sigma_{2})^{i_{2}}
 \cdots(\sigma_{n})^{i_{n}}}{i_{1}i_{2}\cdots i_{n}}~(\bmod ~p^{r+1}),
\end{equation*}
where $\sigma_{i}\in \{1,-1\},~i=1,~2,~\cdots,~n$, but it seems much
more complicated.




\begin{thebibliography}{10}
\bibitem{Ji}\label{Ji} C. G. Ji, A simple proof of a curious congruence by Zhao, Proc. Amer. Math. Soc., 133(2005): 3469-3472.
\bibitem{W}\label{W} L. Q. Wang, A new super congruence involving multiple harmonic sums, arxiv:1410.1712.
\bibitem{W1}\label{W1} L. Q. Wang, A curious congruence involving
alternating  harmonic sums, J. Comb. Number Theory, accepted.
\bibitem{WC}\label{WC} L. Q. Wang, T. X. Cai , A curious congruence
modulo prime powers, J. Number Theory 144(2014): 15-24.
\bibitem{XC}\label{XC} B. Z. Xia, T. X. Cai , Bernoulli numbers and congruences for harmonic sums, Int. J. Number Theory, 6(2010):
849-855.
\bibitem{Zh1}\label{Zh1}   J. Q. Zhao, A super congruence involving multiple harmonic sums, arxiv:1404.3549.
\bibitem{Zh}\label{Zh}   J. Q. Zhao, Bernoulli Numbers, Wolstenholme¡¯s Theorem,
and $p^5$ Variations of Lucas¡¯ Theorem, J. Number Theory
123(2007): 18-26.
\bibitem{ZC}\label{ZC} X. Zhou, T. X. Cai , A Generalization of a Curious Congruence on
Harmonic Sums. Proc. Amer. Math. Soc., 135 (2007): 1329-1333.
\end{thebibliography}
\end{document}